\documentclass[a4paper,12pt,reqno]{amsart} 

\usepackage{amssymb,latexsym}
\usepackage{amsmath,amsthm}
\usepackage{enumerate}
\usepackage[mathscr]{eucal}
\usepackage{subcaption}
\usepackage{wrapfig}
\usepackage{graphicx}
\usepackage{multicol}
\usepackage{url}

\usepackage{graphics} 

\theoremstyle{plain}
\newtheorem{theorem}{\indent Theorem}[section]

\newtheorem{lemma}[theorem]{\indent Lemma}
\theoremstyle{definition}
\newtheorem{definition}[theorem]{\indent Definition}
\newtheorem{example}[theorem]{\indent Example}
\theoremstyle{remark}
\newtheorem{remark}[theorem]{\indent Remark}

\allowdisplaybreaks

\newcommand{\sfour}{\textup{$\mathbf{S4}$}\thinspace}
\newcommand{\cmtf}{\textup{$\mathbf{CMT4}$}\thinspace}

\newcommand{\maxx}{\textup{\emph{M}\thinspace}}
\newcommand{\tax}{\textup{\emph{T}\thinspace}}
\newcommand{\cax}{\textup{\emph{C}\thinspace}}
\newcommand{\kax}{\textup{\emph{K}\thinspace}}
\newcommand{\four}{\textup{\emph{4}\thinspace}}
\newcommand{\nax}{\textup{\emph{N}\thinspace}}
\newcommand{\cpc}{\textup{$\mathbf{CPC}$}\thinspace}

\newcommand{\nec}{\textup{\emph{RN}\thinspace}}
\newcommand{\monot}{\textup{\emph{RM}\thinspace}}
\newcommand{\mpon}{\textup{\emph{MP}\thinspace}}
\newcommand{\rext}{\textup{\emph{RE}\thinspace}}

\newcommand{\gitop}{\textup{\textbf{gITop}\thinspace}}
\newcommand{\gitlog}{\textup{\textbf{gITLog}\thinspace}}

\begin{document}
\null
\vskip 1.8truecm

\title[Infra-topologies revisited: logic and clarification of basic notions]
{Infra-topologies revisited: \\ logic and clarification of basic notions} 
\author{Tomasz Witczak}    
\address{Institute of Mathematics\\ Faculty of Science and Technology
\\ University of Silesia\\ Bankowa~14\\ 40-007 Katowice\\ Poland}
\email{tm.witczak@gmail.com} 


\begin{abstract}In this paper we adhere to the definition of \emph{infra-topological} space as it was introduced by Al-Odhari. Namely, we speak about families of subsets which contain $\emptyset$ and the whole universe $X$, being at the same time closed under finite intersections (but not necessarily under arbitrary or even finite unions). This slight modification allows us to distinguish between new classes of subsets (infra-open, ps-infra-open and i-genuine). Analogous notions are discussed in the language of closures. The class of minimal infra-open sets is studied too, as well as the idea of generalized infra-spaces. Finally, we obtain characterization of infra-spaces in terms of modal logic, using some of the notions introduced above.
\end{abstract}



\maketitle


\section{Introduction}

Generalizations of the ordinary definition of topological space are not new invention. Already in 1940s Choquet \cite{CHOQ} introduced \emph{pre-topologies}. Later, in 60s, Levine recognized several classes of sets which have weaker properties than standard open sets (see \cite{PIEK}). Among them were $\alpha$-, semi-, pre-, b- and $\beta$-open sets. In 80s Masshour \cite{MASS} defined \emph{supra-topological} spaces. In the next decade Cs\'{a}sz\'{a}r started systematic study of families closed only under arbitrary unions (see \cite{CSASZ}). This study has been continued by many authors from all over the world. \footnote{Actually, supra-topologies are Cs\'{a}sz\'{a}r's spaces with open universe $X$.} 

In fact, the whole research direction has flourished in the past two decades. Nowadays we have such concepts as reduced topology, peri-topology, minimal structure, weak structure and, finally, generalized weak structure. The last has been defined simply as an arbitrary family of subsets, see \cite{AVILA}. Almost each of these frameworks is equipped with notions analogous to those of continuity, convergence, filter, density, compactness, connectedness or even topological group. 

This line of research can be useful in classification of finite objects. This is, by no means, important in computer science and its applications (like data science, pattern recognition and image classification, see \cite{IRIS}). For example, generalized topologies (in the sense of Cs\'{a}sz\'{a}r)  appear in formal concept analysis and data clustering as \emph{extensional abstractions} (see \cite{SOLD}) or \emph{knowledge spaces}. Interestingly, they have also other applications: even in Banach games and the issue of entropy (see \cite{PAWLAK1}, \cite{PAWLAK2}). 

As for the infra-topological spaces, they have been introduced (under this name) by Al-Odhari in \cite{ODH}. Then they were investigated by the same author in \cite{CONT}. Later, there appeared papers by other authors too (e.g. \cite{DHANA}, \cite{VAIY}). It is worth to say that the notion of "infra-topology" was used also in \cite{CHAKR} but probably with a different meaning (these authors assumed to work "with a pair $(X,T)$ where $X$ is a non-empty set and $T$ is a family of subsets of $X$ which includes $X$"). 

Unfortunately, it seems that original definitions and theorems of Al-Odhari suffer from certain inexactness. Moreover, there are some flaws which should be corrected. This will be shown in the next sections. For this reason, we want to clarify basic notions of infra-topology. Our aim is to connect these spaces with formal logic by means of modal operators. In some sense, this part of our paper is crucial. In fact, we think that various generalizations of topological spaces (like those mentioned above) are still underestimated as semantical tools. The same can be told about topologies in "alternative" universes (e.g. intuitionistic topological spaces or n-ary topologies, see \cite{SEETHA}). 

\section{Openess and interior}

\subsection{Infra-open sets}

In general, the following definition of infra-topological space is taken from \cite{ODH}. However, there is a kind of confusion there. The author writes about "the intersection of the elements of \emph{any} subcollection of $\tau_{iX}$" and states that "any arbitrary intersection of infra-open set is infra-open set". On the other hand, his formal definition in terms of mathematical symbols clearly refers to the finite case. Thus, we propose to assume that infra-topology is basically closed under \emph{finite} intersections. 

\begin{definition}
Let $X$ be an arbitrary set. An \emph{infra-topological} space on $X$ is a collection $\tau_{iX}$ of subsets of $X$ such that:

\begin{enumerate}
\item $\emptyset, X \in \tau_{iX}$.
\item If for any $i$, $1 \leq i \leq n$, $A_i \in \tau_{iX}$, then $\bigcap A_{i} \in \tau_{iX}$. 
\end{enumerate}

If $A \in \tau_{iX}$, then we say that $A$ is an \emph{infra-open} set. 
\end{definition}

We think that it is reasonable to highlight infra-topologies closed under \emph{any} intersections by naming them \emph{Alexandrov} infra-topologies. Also we can speak about \emph{generalized} infra-topological spaces: those which may not contain $X$.

In general, infra-topologies can be identified with \emph{minimal structures} (see \cite{POPA}) closed under finite intersections. 

Below we list several examples of our spaces. First three are taken from \cite{ODH}, the rest is our own invention.

\begin{example}
\label{jeden}
The following spaces are infra-topological (and not topological):

\begin{enumerate}

\item $X = \{a, b, c\}$, $\tau_{iX} = \{\emptyset, X, \{a\}, \{b\}\}$. Note that $\{a\} \cup \{b\} = \{a, b\} \notin \tau_{iX}$. 

\item $X = \{a, b, c, d\}$, $\tau_{iX} = \{\emptyset, X, \{a\}, \{c\}, \{a, b\}, \{a, c\}\}$. Note that, for example, $\{a, b\} \cup \{a, c\} = \{a, b, c\} \notin \tau_{iX}$. 

\item $X = \{a, b, c, d\}$, $\tau_{iX} = \{\emptyset, X, \{c\}, \{d\}, \{b, c\}, \{c, d\}\}$.

\item $X = \{a, b, c, d\}$, $\tau_{iX} = \{\emptyset, X, \{a\}, \{b\}, \{c\}, \{a, b\}, \{a, b, c\}\}$. Note that, for example, $\{a\} \cup \{b\} = \{a, b\} \in \tau_{iX}$ but $\{b\} \cup \{c\} = \{b, c\} \notin \tau_{iX}$.

\item $X = \{a, b, c, d\}$, $\tau_{iX} = \{\emptyset, X, \{b\}, \{a, b\}, \{b, c\}\}$. 

\item $X = \mathbb{N}$ and we assume that the only infra-open sets are $\mathbb{N}$ and those finite subsets of $\mathbb{N}$ whose cardinality does not exceed certain constant $k \in \mathbb{N}$. Clearly, any intersection of such sets also has cardinality $\leq k$ (hence this is Alexandrov infra-topology), while we may easile find two sets whose union has cardinality greater than $k$. Moreover, we may assume that $\mathbb{N} \notin \tau_{iX}$ to obtain generalized version of this space.

\item $X = \mathbb{R}$ and we assume that the only infra-open sets are $\mathbb{R}$ and those subsets of $\mathbb{R}$ whose length (Lebesgue measure) does not exceed certain constant $y \in \mathbb{R}$. Again, we have Alexandrov infra-topology which can be easily transformed into the generalized one. 

\item $X = \mathbb{Z}$ and $\tau_{iX} = \{\emptyset, \mathbb{Z}, \mathbb{Z}^{-}, \mathbb{Z}^{+}\}$. Note that $\mathbb{Z}^{-} \cap \mathbb{Z}^{+} = \emptyset \in \tau_{iX}$ but $\mathbb{Z}^{-} \cup \mathbb{Z}^{+} = \mathbb{Z} \setminus \{0\} \notin \tau_{iX}$. Of course we can replace integers by rationals or reals in this example. Moreover, we can assume that $\{0\}$ is among infra-open sets.

\item $X = [-1, 1]$ and $\tau_{iX} = \{\emptyset, X, [-\frac{1}{4}, \frac{1}{4}], [0, \frac{1}{4}], [0, \frac{1}{2}]\}$.

\item $X = \mathbb{R}$ and we assume that the only infra-open sets are $\mathbb{R}$, $\emptyset$, $\{0\}$, intervals of the form $(-\infty, a]$, $a \leq 0$ and intervals of the form $[b, +\infty)$, $b \geq 0$. Clearly, the intersection of two such intervals has the proper form (it may be empty). On the other hand, their union may be beyond the family (take $(-\infty, -1] \cup [1, +\infty)$). Also we can replace $\leq$ and $\geq$ in the definition above by strict inequalities (in this case we may assume that $\{0\}$ is not infra-open).

\end{enumerate}
\end{example}

Th. 2.4. in \cite{ODH} is a little bit confusing. It states that the union of two infra-topological spaces (on the same universe $X$) \emph{is} an infra-topological space\footnote{This statement is repeated in \cite{CONT}.}. However, in the next remark we read that "the union of infra-topological spaces may not be infra-topological space, in general". There is even a proper counter-example. Indeed:

\begin{example}
\label{dwa}
Let $X = \{a, b, c, d\}$, $\tau_{iX} = \{\emptyset, X, \{a\}, \{c\}, \{a, b\}, \{a, c\}\}$ and $\mu_{iX} = \{\emptyset, X, \{c\}, \{d\}, \{b, c\}, \{c, d\}\}$. 

Then $\tau_{iX} \cup \mu_{iX} = \{\emptyset, X, \{a\}, \{c\}, \{d\}, \{a, b\}, \{a, c\}, \{b, c\}, \{c, d\}\}$ \footnote{Al-Odhari wrongly describes this set as $\{\emptyset, X, \{a\}, \{b\}, \{c\}, \{a, b\}, \{a, c\}, \{b, c\}, \{c, d\}\}$. It may be just a typo.}. 

Now we see that $\{a, b\} \cap \{b, c\} = \{b\} \notin \tau_{iX} \cup \mu_{iX}$. 
\end{example}

Finally, we can say:

\begin{theorem}
If $\tau_{iX}$ and $\mu_{iX}$ are two infra-topologies on the set $X$, then $\tau_{iX} \cap \mu_{iX}$ is still an infra-topology on $X$, while $\tau_{iX} \cup \mu_{iX}$ not necessarily. 
\end{theorem}

\subsection{Infra-interiors}

Al-Odhari defined infra-interior of $A \subseteq X$ in a standard manner: as "the union of all infra-open sets contained in the set A". Moreover, he states later that this is "the biggest\footnote{He writes: "the smallest" but his intention is obvious.} infra-open set" (contained in $A$). But this is not necessarily true: as we already know, the union of infra-open sets may not be infra-open. Let us consider the following case: 

\begin{example}
\label{trzy}
Let $(X, \tau_{iX})$ be like in Ex. \ref{dwa}. Let us take $B = \{a, b, c\} \notin \tau_{iX}$. Clearly, $\{a\} \cup \{c\} \cup \{a, b\} \cup \{a, c\} = B$. However, $B$ is not the biggest infra-open set contained in itself because it is not infra-open at all. 
\end{example}

For the reasons above, we propose the following two definitions:

\begin{definition}
Let $(X, \tau_{iX})$ be an infra-topological space and $A \subseteq X$. Then we define \emph{infra-interior} (i-interior) of $A$ as: 

$iInt(A) = \bigcup\{O \subseteq X: O \in \tau_{iX}, O \subseteq A\}$. 
\end{definition}

\begin{definition}
Let $(X, \tau_{iX})$ be an infra-topological space and $A \subseteq X$. We say that $A$ is \emph{i-genuine} set \emph{iff} $iInt(A)$ is infra-open. The set of all i-genuine sets associated with a given infra-topology $\tau_{iX}$ on $X$ will be named $ig\tau_{iX}$.
\end{definition}

Let us go back to the Ex. \ref{trzy}. In this case, $B$ is not i-genuine. On the other hand, $\{a, b, d\}$ \emph{is} i-genuine because $\{a\} \cup \{a, b\} = \{a, b\} \in \tau_{iX}$. In Ex. \ref{jeden} (8) we have $\mathbb{Z}^{+} \cup \{0\}$ which is i-genuine (but not infra-open). The same can be told about $\mathbb{Z}^{-} \cup \{0\}$.

Not surprisingly, these two lemmas hold:

\begin{lemma}
Let $(X, \tau_{iX})$ be an infra-topological space and $A \subseteq X$ be i-genuine. Then $iInt(A)$ is the biggest (in the sense of inclusion) infra-open set contained in $A$. 
\end{lemma}

\begin{lemma}
\label{cztery}
Let $(X, \tau_{iX})$ be an infra-topological space and $A, B \subseteq X$. Then:

\begin{enumerate} 
\item If $A$ is infra-open, then it is also i-genuine, i.e. $\tau_{iX} \subseteq ig\tau_{iX}$. 

\item If $A$ is infra-open, then $iInt(A) = A$. The converse is not true. 

\item $iInt(A \cap B) = iInt(A) \cap iInt(B)$. 

\item Each singleton is i-genuine.

\item If $iInt(A) = A$, then $A$ is open or is not i-genuine.
\end{enumerate} 
\end{lemma}

\begin{proof}
The proof is simple. As for the third point, it goes just like in topological space. Let us consider singletons. If $\{a\} \in \tau_{iX}$, then it is i-genuine. If not, then it is always true that $\emptyset \subseteq \{a\}$ (and $\emptyset$ is the only infra-open set contained in our singleton). 
\end{proof}

Below we summarize several basic features of i-genuine sets:

\begin{lemma}
In general, the following properties of i-genuine and non-i-genuine sets hold:
\begin{enumerate}
\item The union of two i-genuine sets may not be i-genuine.

\begin{proof}
Take $X = \{a, b, c, d, e\}$ and $\tau_{iX} = \{\emptyset, X, \{a\}, \{b\}, \{c\}, \{a, b\}\}$. Now let us think about $A = \{a, b\}$ and $B = \{c\}$. Both sets belong to $\tau_{iX}$, hence also to $ig\tau_{iX}$. However, $iInt(A \cup B) = \{a, b\} \cup \{c\} = \{a, b, c\} \notin \tau_{iX}$. Thus $M$ is not i-genuine.
\end{proof}

\item The union of two non-i-genuine sets may be i-genuine. 

\begin{proof}
Consider the same infra-space as above. Take $A = \{b, d, c\}$ and $B = \{a, c, d, e\}$. Clearly, $iInt(A) = \{b, c\} \notin \tau_{iX}$ and $iInt(B) = \{a, c\} \notin \tau_{iX}$. However, $A \cup B = X$ and $X$ is i-genuine. 
\end{proof}

\item The intersection of two i-genuine sets is i-genuine. 

\begin{proof}Let $(X, \tau_{iX})$ be an infra-topological space and $A, B \in ig\tau_{iX}$. Then $iInt(A) \in \tau_{iX}$ and $iInt(B) \in \tau_{iX}$. Clearly, $iInt(A \cap B) = iInt(A) \cap iInt(B) \in \tau_{iX}$. Thus $A \cap B$ is i-genuine.

Note that if we restricted the definition of i-genuine sets only to those which have \emph{non-empty} infra-open infra-interior, then this statement would not longer be true. Just consider $A, B$ such that $iInt(A) \neq \emptyset$, $iInt(B) \neq \empty$ but $iInt(A) \cap iInt(B) = \emptyset$.
\end{proof}

\item The intersection of two non-i-genuine sets may be i-genuine.

\begin{proof}
Take $X = \{\emptyset, X, \{a\}, \{b\}, \{a, c\}\}$ and consider $A = \{a, b, c\}$, $B = \{a, b, d\}$. We see that $iInt(A) = \{a, b, c\} \notin \tau_{iX}$, hence $A \notin ig\tau_{iX}$. Also $iInt(B) = \{a, b\} \notin \tau_{iX}$. However, $A \cap B = \{a, b\} \in \tau_{iX} \subseteq ig\tau_{iX}$. 
\end{proof}

\end{enumerate}
\end{lemma}

Note that the third property allows us to say that $ig\tau_{iX}$ is an infra-topology on $X$. Moreover, from the Lem. \ref{cztery} (1) we conclude that $\tau_{iX}$ is coarser than $ig\tau_{iX}$ (or, equivalently, $ig\tau_{iX}$ is finer than $\tau_{iX}$). On the other hand, the same fact allows us to point out another subclass: the one of \emph{strictly i-genuine} sets, namely such that $iInt(A)$ is both infra-open and non-empty. 

On the base of Lem. \ref{cztery} (5) we introduce another subclass of sets associated with a given $\tau_{iX}$: 

\begin{definition}
Let $(X, \tau_{iX})$ be an infra-topological space and $A \subseteq X$. We say that $A$ is \emph{pseudo-infra-open} (ps-infra-open) \emph{iff} $iInt(A) = A$. The set of all ps-infra-open sets associated with $\tau_{iX}$ will be named $p\tau_{iX}$. 
\end{definition}

Clearly, each infra-open set is also ps-infra-open. In fact, this definition is an infra-topological version of the notion presented by Chakrabarti and Dasgupta (see \cite{DASG}) in the context of minimal structures. Of course, for any $A \subseteq X$, $iInt(A) \in p\tau_{iX}$. 

Now, if $A, B \in p\tau_{iX}$, then $iInt(A \cap B) = iInt(A) \cap iInt(B) = A \cap B$. Moreover, we can prove that $iInt(\bigcup \mathcal{A}) = \bigcup \mathcal{A}$, if $\mathcal{A}$ is a family of ps-infra-open sets. The first inclusion, namely $\subseteq$, is obvious. As for the second one, $\supseteq$, assume that $x \in \bigcup \mathcal{A}$, but $x \notin iInt(\bigcup \mathcal{A})$. Hence, $x$ belongs to certain ps-infra-open set $A \in \mathcal{A}$, but for any infra-open set $G$ such that $G \subseteq \bigcup \mathcal{A}$, $x \notin G$. We know that $A = iInt(A)$, so $x$ is in some infra-open $B \subseteq A$. But such $B$ must be contained in $\bigcup \mathcal{A}$. This contradiction, together with the first part of this paragraph, allows us to say that $p\tau_{iX}$ is a topology on $X$ (of course $iInt(\emptyset) = \emptyset$ and $iInt(X) = X$).

Another class which may be interesting, is a class of minimal infra-open sets. Following \cite{CARPIN}, we introduce the definition below:

\begin{definition}
Let $(X, \tau_{iX})$ be an infra-topological space. We say that $A \in \tau_{iX}$ is minimal-infra-open \emph{iff} for any $B \in \tau_{iX}$ $A \cap B = \emptyset$ or $A \subseteq B$. 
\end{definition}

Let us go back to the Ex. \ref{jeden} (8) but with the assumption that $\{0\}$ is infra-open. Now $\{0\}$ is minimal: first, it is contained in $\mathbb{Z}$ (and in itself), second, it has empty intersections with other infra-open sets, namely $\emptyset$, $\mathbb{Z}^{-}$ and $\mathbb{Z}^{+}$. Now think about Ex. \ref{jeden} (10). Here $\{0\}$ is also minimal. Assume for the moment that $\{0\}$ is still among infra-open sets but our intervals are open, i.e. $a < 0$, $b > 0$. Again, $\{0\}$ is minimal - but now it has empty intersection with any infra-open set different than $\mathbb{R}$.

\section{Closeness and closure}

\subsection{Infra-closed sets}

Following Al-Odhari, we introduce the definition below:

\begin{definition}
Let $(X, \tau_{iX})$ be an infra-topological space. A subset $C \subseteq X$ is called \emph{infra-closed} (in $X$) if $X \setminus C = -C \in \tau_{iX}$, e.g. $-C$ is infra-open. 

For brevity, $c\tau_{iX}$ denotes the set of all infra-closed sets associated with $\tau_{iX}$. 
\end{definition}

In Th. 3.1. \cite{ODH} the author states that "any arbitrary finite intersection of infra-closed sets is an infra-closed set". This is not true, as the following counter-example shows:

\begin{example}
Let $(X, \tau_{iX})$ be like in Ex. \ref{jeden} (2). Now $\{c, d\}$ and $\{b, d\}$ are infra-closed (because $\{a, b\}$ and $\{a, c\}$ are infra-open). However, $\{c, d\} \cap \{b, d\} = \{d\} \notin c\tau_{iX}$. 
\end{example}

Contrary to the opinion of Al-Odhari, the following theorem holds:

\begin{theorem}
Let $(X, \tau_{iX})$ be an infra-topological space. Let $A, B \in c\tau_{iX}$. Then $A \cup B \in c\tau_{iX}$, i.e. the finite union of infra-closed sets is infra-closed.
\end{theorem}

\begin{proof}
Consider $-(A \cup B) = -A \cap -B$. The intersection of infra-open sets is open, hence $-(A \cup B) \in \tau_{iX}$. Hence, $A \cup B \in c\tau_{iX}$. 
\end{proof}

In fact, this result can be derived from the Th. 2.4. in \cite{JAMUNA}. The authors write about generalized weak structures closed under finite intersections. \\

\subsection{Infra-closure}

What about the notion of closure? In \cite{ODH} infra-closure of $A$ is defined as the intersection of all infra-closed sets containing $A$. Again, contrary to the remark given in that paper, such closure is not identical with the \emph{smallest} infra-closed set containing $A$:

\begin{example}
\label{piec}
Let $(X, \tau_{iX})$ be like in Ex. \ref{jeden} (2). Let us consider the set $\{d\}$. It is contained in the following infra-closed sets: $\{c, d\}$, $\{b, d\}$, $\{a, b, d\}$, $\{b, c, d\}$ and $X$. Clearly, the intersection of these sets is just $\{d\}$. But $\{d\}$ is not infra-closed (nor infra-open, as we can add).
\end{example}

Hence, we suggest the following two definitions (the first one is quoted from Al-Odhari):
 
\begin{definition}
Let $(X, \tau_{iX})$ be an infra-topological space and $A \subseteq X$. Then we define infra-closure of $A$ as:

$iCl(A) = \bigcap\{C \subseteq X: C \in c\tau_{iX}, A \subseteq C\}$.

\end{definition}

\begin{definition}
Let $(X, \tau_{iX})$ be an infra-topological space and $A \subseteq X$. We say that $A$ is \emph{c-genuine} set \emph{iff} $iCl(A)$ is infra-closed. The set of all c-genuine sets associated with a given infra-topology $\tau_{iX}$ on $X$ will be named $cg\tau_{iX}$.
\end{definition}

Let us go back to the preceding example. As we have shown, $iCl(\{d\}) = \{d\} \notin c\tau_{iX}$, hence $\{d\}$ is not c-genuine (by the way, we see that singletons need not belong to $cg\tau_{iX}$). On the other hand, $\{a, b\}$ \emph{is} c-genuine. Note that the only infra-closed sets in which $\{a, b\}$ is contained are $\{a, b, d\}$ and $X$. Their intersection is $\{a, b, d\}$ and $\{a, b, d\} \in c\tau_{iX}$ (because $\{c\} \in \tau_{iX}$). 

Again, we have two lemmas:

\begin{lemma}
Let $(X, \tau_{iX})$ be an infra-topological space and $A \subseteq X$ be c-genuine. Then $iCl(A)$ is the smallest (in the sense of inclusion) infra-closed set containing $A$. 
\end{lemma}

\begin{lemma}
Let $(X, \tau_{iX})$ be an infra-topological space and $A, B \subseteq X$. Then:

\begin{enumerate} 
\item If $A$ is infra-closed, then it is also c-genuine.

\item If $A$ is infra-closed, then $iCl(A) = A$. The converse is not true.

\item $iCl(A \cup B) = iCl(A) \cup iCl(B)$. 

\begin{proof}
See Th. 2.4. a) in \cite{JAMUNA}.
\end{proof}

\item $iCl(A \cap B) \subseteq iCl(A) \cap iCl(B)$. 

\begin{remark}There  no equality here. Take $(X, \tau_{iX})$ from Ex. \ref{jeden} (4) and consider $A = \{b\}$, $B = \{c\}$. Now $iCl(A) = \{b, d\}$ and $iCl(B) = \{c, d\}$. Hence, $iCl(A) \cap iCl(B) = \{d\}$. However, $A \cap B = \emptyset$ and $iCl(A \cap B) = \emptyset$. \end{remark}

\item If $iCl(A) = A$, then $A$ is infra-closed or is not c-genuine.
\end{enumerate} 
\end{lemma}

The next lemma is about unions and intersections:

\begin{lemma}
In general, the following properties of c-genuine and non-c-genuine sets hold:
\begin{enumerate}

\item The union of two c-genuine sets is c-genuine. 

\begin{proof}
Let $(X, \tau_{iX})$ be an arbitrary infra-space. Assume that $A, B \in cg\tau_{iX}$. Thus, $A \subseteq iCl(A) \in c\tau_{iX}$ and $B \subseteq iCl(B) \in c\tau_{iX}$. Hence, $A \cup B \subseteq iCl(A) \cup iCl(B) \in c\tau_{iX}$ (as we know, the union of infra-closed sets is infra-closed). Now assume that $iCl(A) \cup iCl(B)$ is \emph{not} an intersection of \emph{all} infra-closed sets containing $A \cup B$. Hence, there is $G \in c\tau_{iX}$ such that $A \cup B \subseteq G$ but $iCl(A) \cup iCl(B) \nsubseteq G$. This means that $iCl(A) \nsubseteq G$ or $iCl(B) \nsubseteq G$. Without loose of generality, assume that $iCl(A) \nsubseteq G$. But $A \subseteq A \cup B \subseteq G$ and $iCl(A)$ is an intersection of all infra-closed sets containing $A$. 
\end{proof}

\item The union of two non-c-genuine sets may be c-genuine.

\begin{proof}
Take $X = \{a, b, c\}$ and $\tau_{iX} = \{\emptyset, X, \{a\}, \{b\}, \{c\}\}$. Now $c\tau_{iX} = \{\emptyset, X, \{b, c\}, \{a, c\}, \{a, b\}\}$. $\{a\}$ and $\{b\}$ are not c-genuine but $\{a, b\} \in c\tau_{iX} \subseteq cg\tau_{iX}$. 
\end{proof}

\item The intersection of two c-genuine sets may not be c-genuine. 

\begin{proof}
Take Ex. \ref{jeden} (3). $\{a, b, c\}$ and $\{a, d\}$ are c-genuine (in fact, they are even infra-closed) but $\{a\}$ is not.
\end{proof}

\item The intersection of two non-c-genuine sets may be c-genuine.

\begin{proof}
Consider Ex. \ref{jeden} (4). We have $c\tau_{iX} = \{\emptyset, X, \{b, c, d\}, \{a, c, d\},\\ \{a, b, d\}, \{c, d\}, \{d\}\}$. Now $\{a, d\}$ and $\{b, d\}$ are non-c-genuine but $\{d\}$ is infra-closed, hence c-genuine.
\end{proof}

\end{enumerate}
\end{lemma}

\emph{Per analogiam} with pseudo-infra-openess we may introduce the following definition:

\begin{definition}
Let $(X, \tau_{iX})$ be an infra-topological space and $A \subseteq X$. We say that $A$ is \emph{pseudo-infra-closed} (ps-infra-closed) \emph{iff} $iCl(A) = A$. The set of all ps-infra-closed sets associated with $\tau_{iX}$ will be named $pc\tau_{iX}$. 
\end{definition}

Clearly, for any $A \subseteq X$, $iCl(A) \in pc\tau_{iX}$. Assume now that $A, B \in pc\tau_{iX}$. We see that $iCl(A \cup B) = iCl(A) \cup iCl(B) = A \cup B$, hence the finite union of ps-infra-closed sets is also ps-infra-closed. Note that $\{d\}$ in Ex. \ref{piec} belongs to $pc\tau_{iX}$.

\section{Infra-topologies and modal logic}

\subsection{Basics}
This section deals with formal logic. Our main result is the one about completeness of the logic named \gitlog with respect to the class of all generalized infra-topological models. These structures are based on generalized infra-spaces but slightly more complex. Basically, we repeat the reasoning from our paper \cite{WITCZAK}. Here it is presented in infra-setting, while originally this solution was designed for spaces closed under arbitrary unions (and not necessarily for finite intersections). What is really important, is the fact that we speak about \emph{generalized} spaces. For this reason, some points (possible worlds, if we speak about logical model) are beyond any open set. However, we can connect each such point with a point having open neighbourhood. This allows us to discuss two modalities, namely two necessities $\Box$ and $\blacksquare$ (the first one is stronger). In general, the idea of two necessities has been studied by some authors: for example, Bo\v{z}i\'c and Do\v{s}en in \cite{BOZIC}. These authors provided a system with $\Box_{S4}$ and $\Box_{K}$ (necessities taken from two normal logics). It is possible to show that this system can be translated into the normal \emph{intuitionistic} modal logic with axioms \kax and \tax. This line of reasoning has been recognised also by us in \cite{WITCZAK2} (in the neighborhood framework). However, our infra-topological logic will not be normal, it will belong to the vast area of weak modal logics. Our conjecture is that it can be connected in some way with subintuitionistic modal logics.

Technical details are presented below. 

\subsection{Alphabet and language}
Our language is propositional, i.e. without quantifiers. $PV$ is a fixed and denumerable set of propositional variables $p, q, r, s, ...$. Logical connectives and operators are $\land$, $\lor$, $\rightarrow$, $\bot$, $\lnot$, $\Box$ and $\blacksquare$. Formulas are generated recursively: if $\varphi$, $\psi$ are \emph{well-formed formulas} then $\varphi \lor \psi$, $\varphi \land \psi$, $\varphi \rightarrow \psi$, $\Box \varphi$ and $\blacksquare \varphi$ are \emph{wff's} too. We shall work with the following list of modal axioms and rules \footnote{Actually, it would be better to speak about \emph{schemes} of axioms and rules.} (they are presented with respect to $\Box$; later we shall use subscripts to distinguish between $\Box$- and $\blacksquare$-versions):

\begin{multicols}{2}
\begin{itemize}

\item \maxx: $\Box(\varphi \land \psi) \rightarrow \Box \varphi \land \Box \psi$
\item \cax: $\Box \varphi \land \Box \psi \rightarrow \Box(\varphi \land \psi) $
\item \tax: $\Box \varphi \rightarrow \varphi$
\item \four: $\Box \varphi \rightarrow \Box \Box \varphi$
\item \nax: $\Box \top$ (\emph{truth axiom})
\item \rext: $\varphi \leftrightarrow \psi \vdash \Box \varphi \leftrightarrow \Box \psi$ (rule of extensionality)
\item \nec: $\varphi \vdash \Box \varphi$ (rule of necessity)
\item \monot: $\varphi \rightarrow \psi \vdash \Box \varphi \rightarrow \Box \psi$ (rule of monotonicity)
\item \mpon: $\varphi, \varphi \rightarrow \psi \vdash \psi$ (\emph{modus ponens})

\end{itemize}
\end{multicols}

The notion of syntactic consequence is typical for modal logic: if $w$ is a set of \gitlog-formulas, then $w \vdash \varphi$ \emph{iff} $\varphi$ can be obtained from the finite subset of $w$ by using axioms of $\gitlog$ and rule \mpon. Clearly, if $\varphi \in w$, then $w \vdash \varphi$. 

\subsection{The notion of model} 

In the next point we define generalized infra-topological models:

\begin{definition}
We define generalized infra-topological model (\gitop-model) as a quintuple $M = \langle W, \tau_{iW}, \mathbf{f}, \mathcal{N}, V \rangle$ where $\tau_{iW}$ is a generalized infra-topology on $W$, $V$ is a function from $PV$ into $P(W)$ and $W$ consists of two separate subsets $Y_1$ and $Y_2$:

\begin{enumerate}
\item If $w \in Y_1$, then we link $w$ with certain $v \in \bigcup \tau_{iW}$ (by means of a function $\mathbf{f}$, i.e. $\mathbf{f}$ is a function from $Y_1$ into $\bigcup \tau_{iW}$).

\item If $w \in Y_2$ then we associate $w$ with certain family $\mathcal{N}(w) = \mathcal{N}_{w} \subseteq P(P(W))$, hence $\mathcal{N}$ is a function from $Y_2$ into $P(P(W))$.

\end{enumerate}
\end{definition}

Now we define forcing of complex formulas:

\begin{definition}
If $M = \langle W, \tau_{iW}, \mathbf{f}, \mathcal{N}, V \rangle$ is an \gitop-model, then we define relation $\Vdash$ between worlds and formulas as below:

\begin{enumerate}
\item $w \Vdash q \Leftrightarrow w \in V(q)$ for any $q \in PV$.
\item $w \Vdash \varphi \land \psi$ (resp. $\varphi \lor \psi$) $\Leftrightarrow w \Vdash \varphi$ and (resp. or) $w \Vdash \psi$.
\item $w \Vdash \varphi \rightarrow \psi \Leftrightarrow w \nVdash \varphi$ or $w \Vdash \psi$.
\item $w \Vdash \lnot \varphi \Leftrightarrow w \nVdash \varphi$.
\item $w \Vdash \Box \varphi \Leftrightarrow \text{ there is } X \in \tau_{iW} \text{ such that } w \in X \text{ and for each } v \in X, v \Vdash \varphi$. 
\item $w \Vdash \blacksquare \varphi \Leftrightarrow$

\begin{enumerate}

\item There is $X \in \tau_{iW}$ such that $\mathbf{f}(w) \in X$ and for any $v \in X$, $v \Vdash \varphi$; \quad \emph{iff} $w \in Y_1$.

\item $V(\varphi) \in \mathcal{N}_{w}$; \quad \emph{iff} $w \in Y_2$. 

\end{enumerate}
\end{enumerate}
\end{definition}

We say that $\varphi$ is \emph{satisfied in a given world} $w$ if $w \Vdash \varphi$. It is \emph{true in a given model} if it is satisfied in each of its worlds. Finally, it is \emph{tautology} if it is true in each infra-topological model. 

Our logic \gitlog is defined as the following system of axiom schemes and rules: $\cpc \cup \{\cax_{\Box}, \maxx_{\Box}, \tax_{\Box}, \four_{\Box}, \rext_{\Box}, \rext_{\blacksquare}, \mpon\}$. By \cpc we understand all modal instances of the classical propositional tautologies. Note that we do not have rule of necessity here (this is because $W$ may not be infra-open). For this reason, we do not identify $\cax_{\Box} \land \maxx_{\Box}$ (i.e. $\Box (\varphi \land \psi) \leftrightarrow \Box \varphi \land \Box \psi$) with well-known axiom $\kax_{\Box}$ (namely, $\Box(\varphi \rightarrow \psi) \rightarrow (\Box \varphi \rightarrow \Box \psi)$. 

It is not difficult to prove theorem below:

\begin{theorem}
Logic \gitlog is sound with respect to the class of all \gitop-models, i.e. each axiom is tautology and each rule holds.
\end{theorem}

We say that any superset of \gitlog is a \emph{theory}. Using standard Henkin method we can extend each consistent (thus not containing $\bot$) theory $w$ to the maximal consistent theory $u$ (maximality means here that for any formula $\varphi$, either $\varphi \in u$ or $\lnot \varphi \in u$, but not both). 

\begin{definition}
We define canonical infra-model as a quintuple $\langle W, \tau_{iW}, \mathbf{f}, \mathcal{N}, V \rangle$, where:

\begin{enumerate}
\item $W$ is a collection of all maximal \gitlog theories. 

\item $\tau_{iW}$ is a generalized infra-topology on $W$ established in the following way: it contains all these subsets of $W$ which can be written as $\widehat{\Box \varphi}$, i.e. $\{z \in W; \Box \varphi \in z\}$ (for certain $\varphi$).

\item $Y_1$ is a set of all such theories from $W$ for which there is a theory $u \in \bigcup \tau_{iW}$ such that for any formula $\varphi$ we have: $\Box \varphi \in u \Leftrightarrow \blacksquare \varphi \in w$.

\item $Y_2 = W \setminus Y_1$. 

\item $\mathbf{f}$ is a function from $Y_1$ into $\bigcup \tau_{iW}$ such that $\mathbf{f}(w) = u$ where $u$ is as in (3).

\item $\mathcal{N}$ is a function from $Y_2$ into $P(P(W))$ defined as: $\mathcal{N}_{w} = \{ \widehat{\varphi}; \blacksquare \varphi \in w\}$. 

\item $V$ is a function from $PV$ into $P(W)$ such that for any $q \in PV$, $V(q) = \{w \in W; q \in W\}$.

\end{enumerate}
\end{definition}

Of course we must prove the following lemma:

\begin{lemma}
Generalized canonical infra-model is indeed \gitop-model. 
\end{lemma}

\begin{proof}
We must prove that $\tau_{iW}$ is a generalized infra-topology on $W$. Clearly, $\emptyset$ can be written as $\widehat{\Box \bot}$. Now let as assume that $A, B \in \tau_{iW}$ and consider $A \cap B$. Of course $A = \widehat{\Box \varphi}$ for certain $\varphi$ and $B = \widehat{\Box \psi}$ for certain $\psi$. Then $A \cap B = \{z \in W; \Box \varphi \in z \text{ and } \Box \psi \in z\}$. But then we use axioms $\cax_{\Box}$ and $\maxx_{\Box}$ to say that $A \cap B = \{z \in W; \Box(\varphi \land \psi) \in z\}$. 

\end{proof}

Using Lindenbaum theorem and rule of $\blacksquare$-extensionality, we can prove the following two lemmas (typical for neighborhood semantics, see \cite{PACUIT}):

\begin{lemma}
Let $W$ be a collection of all maximal theories of \gitlog and let $\{z \in W; \varphi \in z \} = \{z \in W ; \psi \in z\}$. Then $\varphi \rightarrow \psi \in \gitlog$. 
\end{lemma}

\begin{lemma}
In a canonical \gitop-model we have the following property: for each maximal theory $w$, if $\{z \in W; \varphi \in z\} \in \mathcal{N}_{w}$ and $\{z \in W; \varphi \in z\} = \{z \in W; \psi \in z\}$, then $\blacksquare \psi \in w$. 
\end{lemma}

The next lemma is characteristic for the Henkin method (but also crucial for all considerations about completeness):

\begin{lemma}
In a canonical \gitop-model, for any $\gamma$ and for any maximal consistent theory $w$, we have:

$w \Vdash \gamma \Leftrightarrow \gamma \in w$.
\end{lemma}

\begin{proof}
The proof goes by the induction on the complexity of formulas. We do not deal here with Boolean cases which are standard. 

\begin{enumerate}
\item Assume that $\gamma = \Box \varphi$. 

$(\Leftarrow)$

Suppose that $\Box \varphi \in w$. Hence $w \in \widehat{\Box \varphi}$. This set is infra-open. By means of axiom $\tax$ we have $\widehat{\Box \varphi} \subseteq \widehat{\varphi}$. Hence, we conclude that there is $X = \widehat{\Box \varphi}$ such that for all $v \in X$, $v \in \widehat{\varphi}$, i.e. $\varphi \in v$. By induction hypothesis, $v \Vdash \varphi$. Thus $w \Vdash \Box \varphi$. 

$(\Rightarrow)$

Assume that $w \Vdash \Box \varphi$. Then there is $X \in \tau_{iW}$ such that $w \in X$ and for any $v \in X$, $v \Vdash \varphi$ (by induction hypothesis it means that $\varphi \in v$). $X$ has the form $\widehat{\Box \psi}$ for certain $\psi$. We see that $\widehat{\Box \psi} \subseteq \widehat{\varphi}$. Hence, $\Box \psi \rightarrow \varphi$ is a theorem of \gitlog (if not, then there would exist maximal consistent set containing both $\Box \varphi$ and $\lnot \varphi$, this would mean that $\widehat{\Box \psi} \nsubseteq \widehat{\varphi}$). By the rule of monotonicity, which is true in infra-topological models (with respect to $\Box$ operator), we infer $\Box \Box \psi \rightarrow \Box \varphi$. From axiom \four we have $\Box \psi \rightarrow \Box \varphi$. Then $\widehat{\Box \psi} \subseteq \widehat{\Box \varphi}$. Thus $\Box \varphi \in w$. 

\item Assume that $\gamma = \blacksquare \varphi$. 

$(\Leftarrow)$

Suppose that $\blacksquare \varphi \in w$. We have two options:

\begin{enumerate}

\item $w \in Y_1$. Hence $\Box \varphi \in u = \mathbf{f}(w)$. Thus $u \in \widehat{\Box \varphi}$. This set belongs to $\tau_{iW}$. By the axiom $\tax_{\Box}$ we state that $\widehat{\Box \varphi} \subseteq \widehat{\varphi}$. Now there is $X = \widehat{\Box \varphi}$ such that $X \in \tau_{iW}$, $u \in X$ and for all $v \in X$ we have $\varphi \in v$ which means (by induction hypothesis) that $v \Vdash \varphi$. Thus $w \Vdash \blacksquare \varphi$. 

\item $w \in Y_2$. By the definition of $\mathcal{N}$ we state that $\{z \in W; \varphi \in z\} = \widehat{\varphi} \in \mathcal{N}_{w}$. By induction hypothesis $\widehat{\varphi} = \{z \in W; z \Vdash \varphi\} \in \mathcal{N}_{w}$. Hence $w \Vdash \blacksquare \varphi$. 

\end{enumerate}

$(\Rightarrow)$

Suppose that $w \Vdash \blacksquare \varphi$. Again we have two possibilities:

\begin{enumerate}

\item $w \in Y_1$. There is $X \in \tau_{iW}$ such that $u = \mathbf{f}(w) \in X$ and for any $v \in X$, $v \Vdash \varphi$. $X$ can be written as $\widehat{\Box \psi}$ (for certain formula $\psi$). As earlier, we prove that $u \in \widehat{\Box \psi} \subseteq \widehat{\Box \varphi}$. Hence, $\Box \varphi \in u$. But $u = \mathbf{f}(w)$. Then $\blacksquare \varphi \in w$. 

\item $w \in Y_2$. Let us assume that $w \Vdash \blacksquare \varphi$. Hence $\{z \in W; z \Vdash \varphi\} \in \mathcal{N}_{w}$. By induction hypothesis $\{z \in W; \varphi \in z\} \in \mathcal{N}_{w}$. It means that $\blacksquare \varphi \in w$. 

\end{enumerate}
\end{enumerate}

\end{proof}

In conclusion, we have the following result:

\begin{theorem}
\gitlog is (strongly) complete with respect to the class of all \gitop-models.
\end{theorem}
\begin{proof}
Suppose that $w$ is a consistent theory of \gitlog and $w \nvdash \varphi$. In particular $\varphi \notin w$. We extend $w$ to the maximal theory $v$ such that $w \subseteq v$ and $\varphi \notin v$. Then for each $\psi \in w$, $v \Vdash \psi$ and $v \nVdash \varphi$. But $w \subseteq v$, hence $\varphi$ is not a semantical consequence of $w$, i.e. $w \nVdash \varphi$. 
\end{proof}

Assume now that our model has been simplified: there is no $Y_1$ and $Y_2$ (but there is still generalized infra-topology $\tau_{iW}$), function $\mathbf{f}$ is absent and we have only one modal operator $\Box$ (defined as earlier). Now we can easily prove that logic \cmtf is complete with respect to this class. Moreover, we can assume that $W$ is infra-open: in this case rule of necessity (or, equivalently, axiom \nax) becomes true, hence our structures correspond to the well-known \sfour logic. However, it is well-known fact that \sfour is complete with respect to the narrower class of models, namely topological models (i.e. with closure under arbitrary unions). 

\section{Final remarks and ruther research}

We think that it would be valuable to investigate the notions of connectedness and density in the context of ps-infra-open and i-genuine sets. As for the connectedness, some attempts were made in \cite{DASG} (but in the setting of minimal structures). As for the density, the following definition seems to be sensible at first glance: $A$ is ps-dense (resp. i-dense) in $X$ \emph{iff} it has non-empty interesection with any non-empty ps-infra-open (resp. i-genuine) subset of $X$. However, in case of i-genuine sets this definition is not valuable: because we know that each singleton is i-genuine, so our dense set would have to have non-empty intersection with \emph{any} singleton. There is only one such set, namely $W$. However, me may demand an intersection with any strictly i-genuine set.

From the logical point of view, it would be interesting to connect these two classes (ps-infra-open and i-genuine sets) with forcing of modalities. Of course, we can easily define new necessity operator, say $\bullet$, in the following manner: $w \Vdash \bullet \varphi \Leftrightarrow$ there is ps-infra-open (resp. i-genuine) set $X$ such that $w \in X$ and for any $v \in X, v \Vdash \varphi$. However, it is not so obvious how to capture this notions in canonical model (clearly, our aim is to establish completeness). Also, one could think about logical characterization of particular infra-spaces, like those from Ex. \ref{jeden} (6) - (8) and (10).

\end{document}